\documentclass[11pt]{article}

\usepackage{amsmath, amssymb, amsthm, graphicx, color, xspace}
\usepackage{marginnote}
\usepackage{hyperref}
\usepackage{booktabs}
\usepackage{algpseudocode}
\usepackage{mdframed}
\usepackage{alltt}
\usepackage{float}
\newfloat{listing}{tbhp}{lst}
\floatname{listing}{Listing}

\newtheorem{thm}{Theorem}
\newtheorem{prop}[thm]{Proposition}
\newtheorem{lem}[thm]{Lemma}

\newcommand{\armh}{L_{\textrm{horiz}}}
\newcommand{\armv}{L_{\textrm{vert}}}

\newcommand{\rotmat}[1]{\textrm{Rot}_{#1}}
\newcommand{\R}{\mathbb{R}}
\newcommand{\alcc}{\lambda^*}

\newcommand{\sofaconst}{\mu_\textrm{MS}}
\newcommand{\sofaconstbeta}{\mu_*}
\newcommand{\gerverconst}{\mu_\textrm{G}}
\newcommand{\romikconst}{\mu_\textrm{R}}
\newcommand{\bestupperbound}{2.37}

\definecolor{varcolor}{RGB}{0,0,160}
\definecolor{funccolor}{RGB}{0,160,0}
\definecolor{codebgcolor}{RGB}{240,240,240}
\definecolor{inputcolor}{RGB}{0,50,160}
\newcommand{\varname}[1]{{\color{varcolor}\textnormal{\texttt{#1}}}}

\newcommand{\keyword}[1]{\textbf{#1}}
\newcommand{\inputline}[1]{{\color{inputcolor}\textnormal{\texttt{#1}}}}

\title{Improved upper bounds in the moving sofa problem}
\author{Yoav Kallus\footnote{Santa Fe Institute, 1399 Hyde Park Road, Santa Fe, NM 87501, USA.} \and Dan Romik\footnote{Department of Mathematics, University of California, Davis, One Shields Ave, Davis, CA 95616, USA. Email: \texttt{romik@math.ucdavis.edu}}}

\begin{document}
\maketitle

\begin{abstract}
The moving sofa problem, posed by L.~Moser in 1966, asks for the planar shape of maximal area that can move around a right-angled corner in a hallway of unit width. It is known that a maximal area shape exists, and that its area is at least $2.2195\ldots$---the area of an explicit construction found by Gerver in 1992---and at most $2\sqrt{2}\approx 2.82$, with the lower bound being conjectured as the true value. We prove a new and improved upper bound of $\bestupperbound$. The method involves a computer-assisted proof scheme that can be used to rigorously derive further improved upper bounds that converge to the correct value.
\end{abstract}

\renewcommand{\thefootnote}{\fnsymbol{footnote}} 
\renewcommand{\thefootnote}{\arabic{footnote}}

\section{Introduction}

The \textbf{moving sofa problem} is a well-known unsolved problem in geometry, first posed by Leo Moser in 1966 \cite{unsolved-problems, moser}. It asks:

\begin{quote}
\textit{What is the planar shape of maximal area that can be moved around a right-angled corner in a hallway of unit width?}
\end{quote}

We refer to a connected planar shape that can be moved around a corner in a hallway as described in the problem as a \textbf{moving sofa shape}, or simply a \textbf{moving sofa}.
It is known \cite{gerver} that a moving sofa of maximal area exists. The shape of largest area currently known is an explicit construction found by Joseph Gerver in 1992  \cite{gerver} (see also \cite{romik} for a recent perspective on Gerver's results), known as \textbf{Gerver's sofa} and shown in Figure~\ref{fig:gerver}. Its area is \textbf{Gerver's constant} 
$$\gerverconst = 2.21953166\ldots,$$ an exotic mathematical constant that is defined in terms of a certain system of transcendental equations but which does not seem to be expressible in closed form.
\begin{figure}
\begin{center}
\scalebox{0.85}{\includegraphics{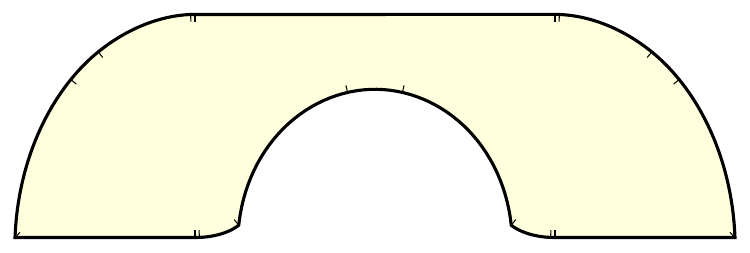}}
\caption{Gerver's sofa, conjectured to be the solution to the moving sofa problem. Its boundary is made up of 18 curves, each given by a separate analytic formula; the tick marks show the points of transition between different analytic pieces of the boundary.}
\label{fig:gerver}
\end{center}
\end{figure}
Gerver conjectured that $\gerverconst$ is the largest possible area for a moving sofa, a possibility supported heuristically by the local-optimality considerations from which his shape was derived. 

Gerver's construction provides a lower bound on the maximal area of a moving sofa. In the opposite direction, it was proved by Hammersley \cite{hammersley} in 1968 that a moving sofa cannot have an area larger than $2\sqrt{2}\approx 2.82$. It is helpful to reformulate these results by denoting
$$ \sofaconst = \max \Big\{ \operatorname{area}(S)\,:\, S \textrm{ is a moving sofa shape} \Big\}, $$
the so-called \textbf{moving sofa constant} (see Finch's book \cite[Sec.~8.12]{finch}; note that Finch refers to Gerver's constant $\gerverconst$ as the ``moving sofa constant,'' but this terminology currently seems unwarranted in the absence of a proof that the two constants are equal.) The above-mentioned results then translate to the statement that
$$ \gerverconst \le \sofaconst \le 2\sqrt{2}. $$

The main goal of this paper is to derive improved upper bounds for $\sofaconst$. We prove the following explicit improvement to Hammersley's upper bound from 1968.

\begin{thm}[New area upper bound in the moving sofa problem] \label{thm:new-upperbound}
The moving sofa constant $\sofaconst$ satisfies the bound
\begin{equation} \label{eq:upperbound} 
\sofaconst \le \bestupperbound.
\end{equation}
\end{thm}

More importantly than the specific bound $\bestupperbound$, our approach to proving Theorem~\ref{thm:new-upperbound} involves the design of a computer-assisted proof scheme that can be used to rigorously derive even sharper upper bounds; in fact, our algorithm can produce a sequence of rigorously-certified bounds that converge to the true value $\sofaconst$ (see Theorems~\ref{thm:conv-moving-sofa} and~\ref{thm:asym-sharpness} below). An implementation of the scheme we coded in \texttt{C++} using exact rational arithmetic certifies $\bestupperbound$ as a valid upper bound after running for 480 hours on one core of a 2.3 GHz Intel Xeon E5-2630 processor. Weaker bounds that are still stronger than Hammersley's bound can be proved in much less time---for example, a bound of $2.7$ can be proved using less than one minute of processing time.

Our proof scheme is based on the observation that the moving sofa problem, which is an optimization problem in an infinite-dimensional space of shapes, can be relaxed in many ways to arrive at a family of finite-dimensional optimization problems in certain spaces of polygonal shapes. These finite-dimensional optimization problems are amenable to attack using a computer search.

Another of our results establishes new restrictions on a moving sofa shape of largest area, and specifically on the angle by which such a shape must rotate as it moves around the corner. Gerver proved \cite[Th.~1]{gerver} that the motion of a largest area moving sofa shape around the corner can be parametrized such that its angle of rotation increases monotonically and continuously from $0$ to some terminal angle $\beta$, with $\pi/3\le \beta\le \pi/2$---that is, a largest area moving sofa must undergo rotation by an angle of at least $\pi/3$ as it moves around the corner, and does not need to rotate by an angle greater than $\pi/2$. As explained in the next section, Gerver's argument actually proves a slightly stronger result with $\pi/3$ replaced by the angle $\beta_0 = \sec^{-1}(\gerverconst) \approx 63.22^\circ$. We will prove the following improved bound on the angle of rotation of a moving sofa of maximal area.

\begin{thm}[New rotation lower bound in the moving sofa problem] \label{thm:angle-bound}
Any moving sofa shape of largest area must undergo rotation by an angle of at least $\sin^{-1}(84/85) \approx 81.203^\circ$ as it moves around the corner.
\end{thm}

There is no reason to expect this bound to be sharp; in fact, it is natural to conjecture that any largest area moving sofa shape must undergo rotation by an angle of $\pi/2$. As with the case of the bound \eqref{eq:upperbound}, our techniques make it possible in principle to produce further improvements to the rotation lower bound, albeit at a growing cost in computational resources.

The paper is arranged as follows. Section~\ref{sec:theory} below defines the family of finite-dimensional optimization problems majorizing the moving sofa problem and develops the necessary theoretical ideas that set the ground for the computer-assisted proof scheme. In Section~\ref{sec:algorithm} we build on these results and introduce the main algorithm for deriving and certifying improved bounds, then prove its correctness. Section~\ref{sec:numerical} discusses specific numerical examples illustrating the use of the algorithm, leading to a proof of Theorems~\ref{thm:new-upperbound} and~\ref{thm:angle-bound}. The Appendix describes \texttt{SofaBounds}, a software implementation we developed as a companion software application to this paper \cite{sofabounds}.

\paragraph{Acknowledgements.} Yoav Kallus was supported by an Omidyar Fellowship at the Santa Fe Institute. Dan Romik thanks Greg Kuperberg for a key suggestion that was the seed from which Proposition~\ref{prop:sofa-fg-bounds} eventually grew, and John Sullivan, Joel Hass, Jes\'us De Loera, Maria Trnkova and Jamie Haddock for helpful discussions. We also thank the anonymous referee for helpful suggestions.

\section{A family of geometric optimization problems}

\label{sec:theory}

In this section we define a family of discrete-geometric optimization problems that we will show are in a sense approximate versions of the moving sofa problem for polygonal regions. Specifically, for each member of the family, the goal of the optimization problem will be to maximize the area of the intersection of translates of a certain finite set of polygonal regions in $\R^2$. It is worth noting that such optimization problems have been considered more generally in the computational geometry literature; see, e.g., \cite{harpeled-roy, mount-silverman}.

We start with a few definitions. Set
\begin{align*}
H &= \R \times [0,1], \\
V &= [0,1] \times \R, \\
\armh &= (-\infty,1]\times[0,1],  \\
\armv &= [0,1]\times (-\infty,1], \\
L_0 &= \armh \cup \armv.
\end{align*}
For an angle $\alpha \in [0,\pi/2]$ and a vector $\mathbf{u}=(u_1,u_2)\in\R^2$, denote
\begin{align*}
L_\alpha(\mathbf{u}) &= 
\Big\{ (x,y)\in\R^2\,:\,
u_1 \le x\cos \alpha + y\sin \alpha \le u_1+1 
\\& \qquad\qquad\qquad\qquad\qquad\qquad \textrm{ and \ } -x\sin\alpha + y \cos\alpha \le u_2+1
\Big\}
\\ & \quad \cup
\Big\{ (x,y)\in\R^2\,:\,
x\cos \alpha + y\sin \alpha \le u_1+1 
\\ & \qquad\qquad\qquad\qquad\qquad\qquad
\textrm{ and } u_2\le -x\sin\alpha + y \cos\alpha \le u_2+1
\Big\},
\end{align*}
For angles $\beta_1, \beta_2$, denote
\begin{align*}
B(\beta_1,\beta_2) &= 
\Big\{ (x,y)\in\R^2\,:\,
    0 \le x\cos \beta_1 + y\sin \beta_1
\\& \qquad\qquad\qquad\qquad\qquad\qquad \textrm{ and \ } x\cos \beta_2 + y\sin \beta_2 \le 1
\Big\}
\\ & \quad \cup
\Big\{ (x,y)\in\R^2\,:\,
    x\cos \beta_1 + y\sin \beta_1 \le 1
\\& \qquad\qquad\qquad\qquad\qquad\qquad \textrm{ and \ } 0 \le x\cos \beta_2 + y\sin \beta_2
\Big\},
\end{align*}
Geometrically, $L_\alpha(\mathbf{u})$ is the $L$-shaped hallway $L_0$ translated by the vector~$\mathbf{u}$ and then rotated around the origin by an angle of $\alpha$; and $B(\beta_1,\beta_2)$, which we nickname a ``butterfly set,'' is a set that contains a rotation of the vertical strip $V$ around the origin by an angle $\beta$ for all $\beta\in[\beta_1,\beta_2]$. See Fig.~\ref{fig:geometric-sets}.

\begin{figure}
\begin{center}
\begin{tabular}{ccc}
 \scalebox{0.5}{\includegraphics{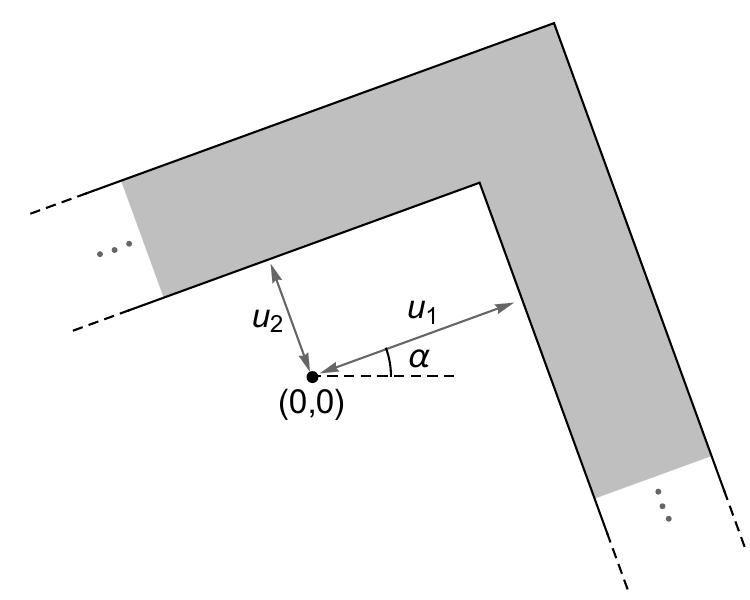}} 
 & \scalebox{0.5}{\includegraphics{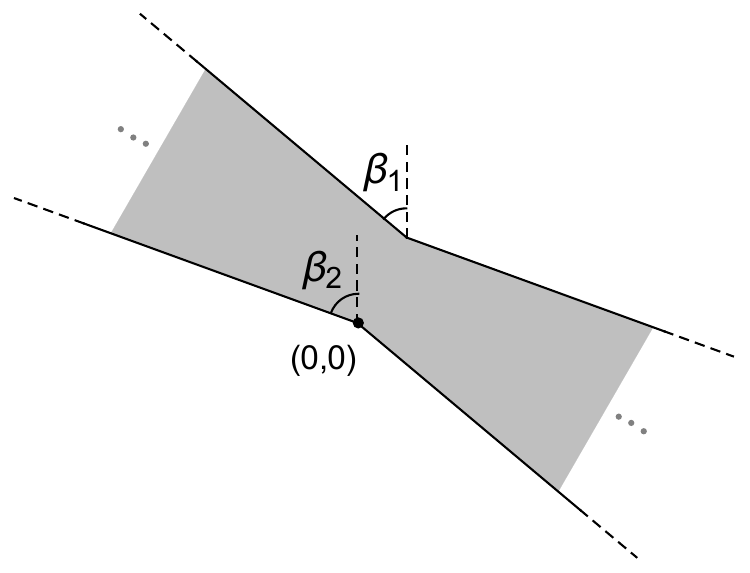}} 
\\
(a) & (b)
\end{tabular}
\caption{
a) The rotated and translated $L$-shaped corridor $L_\alpha(u_1,u_2)$. (b) The ``butterfly set'' $B(\beta_1,\beta_2)$.
}
\label{fig:geometric-sets}
\end{center}
\end{figure}

Next, let $\lambda$ denote the area measure on $\R^2$ and let $\alcc(X)$ denote the maximal area of any connected component of $X\subset\R^2$.
Given a vector $\boldsymbol{\alpha} = (\alpha_1,\ldots,\alpha_k)$ of angles $0< \alpha_1<\ldots < \alpha_k < \pi/2$ and two additional angles $\beta_1,\beta_2 \in (\alpha_k,\pi/2]$ with $\beta_1\le\beta_2$, define
\begin{align}
g_{\boldsymbol{\alpha}}^{\beta_1,\beta_2}(\mathbf{u}_1,\ldots,\mathbf{u}_k) &= \alcc\left(
H \cap \bigcap_{j=1}^k L_{\alpha_j}(\mathbf{u}_j) \cap B(\beta_1,\beta_2) 
\right) \quad (\mathbf{u}_1,\ldots,\mathbf{u}_k \in \R^2), \label{eq:def-little-g} \\[5pt]
\label{eq:def-G}
G_{\boldsymbol{\alpha}}^{\beta_1,\beta_2}
&= \sup \left\{
g_{\boldsymbol{\alpha}}^{\beta_1,\beta_2}(\mathbf{u}_1,\ldots,\mathbf{u}_k) 
\,:\,
\mathbf{u}_1,\ldots,\mathbf{u}_k \in \R^{2}
\right\}. 
\end{align}
An important special case is $G_{\boldsymbol{\alpha}}^{\pi/2,\pi/2}$, which we denote simply as $G_{\boldsymbol{\alpha}}$. Note that $B(\beta_1,\pi/2)\cap H = H$, so in that case the inclusion of $B(\beta_1,\beta_2)$ in \eqref{eq:def-little-g} is superfluous.

The problem of computing $G_{\boldsymbol{\alpha}}^{\beta_1,\beta_2}$ is an optimization problem in $\R^{2k}$. The following lemma shows that the optimization can be performed on a compact subset of $\R^{2k}$ instead.

\begin{lem}\label{lem:supmax}
There exists a box $\Omega_{\boldsymbol{\alpha}}^{\beta_1,\beta_2}=[a_1,b_1]\times\ldots \times [a_{2k},b_{2k}] \subset \R^{2k}$, with the values of $a_i,b_i$ being explicitly computable functions of $\boldsymbol{\alpha}, \beta_1, \beta_2$, such that
\begin{equation} \label{eq:supmax}
G_{\boldsymbol{\alpha}}^{\beta_1,\beta_2} = 
 \max \left\{
g_{\boldsymbol{\alpha}}^{\beta_1,\beta_2}(\mathbf{u}) 
\,:\,
\mathbf{u} \in \Omega_{\boldsymbol{\alpha}}^{\beta_1,\beta_2}
\right\}. 
\end{equation}
\end{lem}

\begin{proof}
We will show that any value of $g_{\boldsymbol{\alpha}}^{\beta_1,\beta_2}(\mathbf{u})$ attained for some $\mathbf{u}\in\R^2$ is matched by a value attained inside a sufficiently large box. This will establish that $g_{\boldsymbol{\alpha}}^{\beta_1,\beta_2}(\mathbf{u})$ is bounded from above; the fact that it attains its maximum follows immediately, since $g_{\boldsymbol{\alpha}}^{\beta_1,\beta_2}(\mathbf{u})$ is easily seen to be an upper semicontinuous function.

   Start by observing that for every interval $[x_1,x_2]$ and $0<\alpha<\pi/2$, there are intervals $I$ and $J$ such
    that if $(u,v)\in \R^2\setminus I\times J$, the set $\big([x_1,x_2]\times[0,1]\big)\cap L_\alpha(u,v)$ is either empty or is identical to $\big([x_1,x_2]\times[0,1]\big)\cap L_\alpha(u',v')$
    for some $(u',v')\in I\times J$. Indeed, this is valid with the choices \begin{align*}
    I&=[x_1 \cos\alpha-1,x_2 \cos\alpha+\sin\alpha], \\
    J&=[-x_2 \sin\alpha-1,-x_1 \sin\alpha+\cos\alpha].
\end{align*}

    We now divide the analysis into two cases. First, if $\beta_2<\pi/2$, then $H\cap B(\beta_1,\beta_2)\subseteq[-\tan \beta_2,\sec \beta_2]\times[0,1]$.
    Therefore, if we define $\Omega_{\boldsymbol{\alpha}}^{\beta_1,\beta_2}=I_1\times J_1 \times I_2 \times J_2 \times \cdots\times I_k \times J_k$, where for each $1\le i\le k$, $I_i$ and $J_i$ are intervals $I,J$ as described in the above observation as applied to the angle $\alpha=\alpha_i$, then $g_{\boldsymbol{\alpha}}^{\beta_1,\beta_2}(u_1,\ldots,u_{2k})$ is guaranteed to
    attain its maximum on $\Omega_{\boldsymbol{\alpha}}^{\beta_1,\beta_2}$, since any value attained outside $\Omega_{\boldsymbol{\alpha}}^{\beta_1,\beta_2}$ is either zero or matched by a value attained inside~$\Omega_{\boldsymbol{\alpha}}^{\beta_1,\beta_2}$.
    
    Second, if $\beta_2=\pi/2$, then $H\cap B(\beta_1,\beta_2) = H$.
    The optimization objective function $g_\mathbf{\alpha}^{\beta_1,\beta_2}(\mathbf{u}_1,\ldots,\mathbf{u}_k)$ is invariant to translating all the rotated $L$-shaped hallways horizontally by the same amount (which corresponds to translating each variable $\mathbf{u}_j$ in the direction of the vector $(\cos\alpha_j, -\sin\alpha_j)$).
    Therefore, fixing an arbitrary $1\le j\le k$, any value of $g_\mathbf{\alpha}^{\beta_1,\beta_2}(\mathbf{u}_1,\ldots,\mathbf{u}_k)$ attained on $\mathbb{R}^{2k}$ is also attained at some point satisfying $\mathbf{u}_j=(0,u_{j,2})$.
    Furthermore, we can constrain $u_{j,2}$ as follows: first, when $u_{j,2}<-\tan\alpha_j-1$, then $L_{\alpha_j}(0,u_{j,2})\cap H$ is empty.
    Second, when $u_{j,2}>\sec\alpha_j$, then $L_{\alpha_j}(0,u_{j,2})\cap H$ is the union of two disconnected components, one of which is a translation
    of $\rotmat{\alpha_j}(H)\cap H$ and the other is a translation of $\rotmat{\alpha_j}(V)\cap H$.
    Since the largest connected component of $H \cap \bigcap_{j=1}^k L_{\alpha_j}(\mathbf{u}_j)$ is contained in one of these two rhombuses, and
    since the translation of the rhombus does not affect the maximum attained area, we see that any objective value attained with $\mathbf{u}_j=(0,u_{j,2})$,
    where $u_{j,2}>\sec\alpha_j$, can also be attained with $\mathbf{u}_j=(0,\sec\alpha_j)$. So, we may restrict $\mathbf{u}_j\in I_j\times J_j$, where
    $I_j = \{0\}$ and $J_j=[-\tan\alpha_j-1,\sec\alpha_j]$.
    Finally, when $\mathbf{u}_j\in I_j\times J_j$, we have $H\cap L_{\alpha_j}(\mathbf{u})\subseteq[\csc\alpha_j,\sec\alpha_j]\times[0,1]$, so we can repeat
    a procedure similar to the one used in the case $\beta=\pi/2$ above to construct intervals $I_i$ and $J_i$ for all $i\neq j$ to ensure that \eqref{eq:supmax} is satisfied.
\end{proof}

We now wish to show that the function $G_{\boldsymbol{\alpha}}^{\beta_1,\beta_2}$ relates to the problem of finding upper bounds in the moving sofa problem. The idea is as follows. Consider a sofa shape $S$ that moves around the corner while rotating continuously and monotonically  (in a clockwise direction, in our coordinate system) between the angles $0$ and $\beta \in [0,\pi/2]$. As we mentioned in the Introduction, a key fact proved by Gerver \cite[Th.~1]{gerver} is that in the moving sofa problem it is enough to consider shapes being moved in this fashion. By changing our frame of reference to one in which the shape stays fixed and the $L$-shaped hallway $L_0$ is dragged around the shape while being rotated, we see (as discussed in \cite{gerver, romik}) that $S$ must be contained in the intersection

\begin{equation} \label{eq:sx-intersections}
S_{\mathbf{x}} = \armh \cap \bigcap_{0\le t\le \beta} L_t(\mathbf{x}(t))
\cap \Big(\mathbf{x}(\beta)+\rotmat{\beta}(\armv) \Big),
\end{equation}
where $\mathbf{x}:[0,\beta]\to\R^2$ is a continuous path satisfying $\mathbf{x}(0)=(0,0)$ that encodes the path by which the hallway is dragged as it is being rotated, and where $\rotmat{\beta}(\armv)$ denotes $\armv$ rotated by an angle of $\beta$ around $(0,0)$ (more generally, here and below we use the notation $\rotmat{\beta}(\cdot)$ for a rotation operator by an angle of $\beta$ around the origin). We refer to such a path as a \textbf{rotation path}, or a \textbf{$\beta$-rotation path} when we wish to emphasize the dependence on $\beta$. Thus, 
the area of $S$ is at most $\alcc(S_\mathbf{x})$, the maximal area of a connected component of $S_\mathbf{x}$, and conversely, a maximal area connected component of $S_\mathbf{x}$ is a valid moving sofa shape of area $\alcc(S_\mathbf{x})$.
Gerver's result therefore implies that
\begin{equation} \label{eq:sofaconst-characterization}
\sofaconst = \sup \Big\{ \alcc(S_\mathbf{x}) \,:\,
\mathbf{x}\textrm{ is a $\beta$-rotation path for some }\beta\in[0,\pi/2]
\Big\}.
\end{equation}
It is also convenient to define
$$
\sofaconstbeta(\beta) = \sup \Big\{ \alcc(S_\mathbf{x}) \,:\,
\mathbf{x}\textrm{ is a $\beta$-rotation path}
\Big\}
\qquad (0 \le \beta \le \pi/2).
$$
so that we have the relation
\begin{equation} \label{eq:sofaconst-beta}
\sofaconst = \sup_{0 < \beta \le \pi/2} \sofaconstbeta(\beta). 
\end{equation}
Moreover, as Gerver pointed out in his paper, $\sofaconstbeta(\beta)$ is bounded from above by the area of the intersection of the horizontal strip $H$ and the rotation of the vertical strip $V$ by an angle $\beta$, which is equal to $\sec(\beta)$. Since $\sofaconst \ge \gerverconst$, and $\sec(\beta)\ge \gerverconst$ if and only if $\beta\in[\beta_0,\pi/2]$, where we define $\beta_0 = \sec^{-1}(\gerverconst) \approx 63.22^\circ$, we see that in fact
\begin{equation} \label{eq:sofaconst-beta0}
\sofaconst = \sup_{\beta_0 \le \beta \le \pi/2} \sofaconstbeta(\beta), 
\end{equation}
and furthermore, $\sofaconst > \sofaconstbeta(\beta)$ for any $0<\beta<\beta_0$, i.e., any moving sofa of maximal area has to rotate by an angle of at least $\beta_0$.
(Gerver applied this argument to claim a slightly weaker version of this result in which the value of $\beta_0$ is taken as $\pi/3 = \sec^{-1}(2)$; see \cite[p.~271]{gerver}).
Note that it has not been proved, but seems natural to conjecture, that $\sofaconst = \sofaconstbeta(\pi/2)$---an assertion that would follow from Gerver's conjecture that the shape he discovered is the moving sofa shape of largest area, but may well be true even if Gerver's conjecture is false.

The relationship between our family of finite-dimensional optimization problems and the moving sofa problem is made apparent by the following result.

\begin{prop}
\label{prop:sofa-fg-bounds}
(i) For any $\boldsymbol{\alpha}=(\alpha_1,\ldots,\alpha_k)$ and $\beta$ with $0<\alpha_1<\ldots<\alpha_k \le \beta\le\pi/2$,
we have
\begin{equation} 
\sofaconstbeta(\beta) \le G_{\boldsymbol{\alpha}}. \label{eq:sofa-fg-bound1}
\end{equation}

\noindent (ii) For any $\boldsymbol{\alpha}=(\alpha_1,\ldots,\alpha_k)$ with $0<\alpha_1<\ldots<\alpha_k \le \beta_0$, we have
\begin{equation}
\sofaconst \le G_{\boldsymbol{\alpha}}. \label{eq:sofa-fg-bound2}
\end{equation} 

\noindent (iii) 
For any $\boldsymbol{\alpha}=(\alpha_1,\ldots,\alpha_k)$ and $\beta_1,\beta_2$ with $0<\alpha_1<\ldots<\alpha_k \le \beta_1 < \beta_2 \le \pi/2$, we have
\begin{equation}
\sup_{\beta\in[\beta_1,\beta_2]} \sofaconstbeta(\beta) \le G_{\boldsymbol{\alpha}}^{\beta_1,\beta_2}.
\label{eq:sofa-fg-bound3}
\end{equation} 
\end{prop}

\begin{proof} 
Start by noting that, under the assumption that  $0 < \alpha_1 <\ldots < \alpha_k \le \beta \le \pi/2$, if $\mathbf{x}$ is a $\beta$-rotation path then the values $\mathbf{x}(\alpha_1),\ldots,\mathbf{x}(\alpha_k)$ may potentially range over an arbitrary $k$-tuple of vectors in $\R^2$. It then follows that
\begin{align*}
\sofaconstbeta(\beta) &= \sup \Big\{ \alcc(S_\mathbf{x}) \,:\,
\mathbf{x}\textrm{ is a $\beta$-rotation path}
\Big\}
\nonumber \\ &=
\sup \left\{ \alcc\left(\armh \cap \bigcap_{0\le t\le \beta} L_t(\mathbf{x}(t))
\cap \Big(\mathbf{x}(\beta)+\rotmat{\beta}(\armv) \Big)\right) \,:\,
\right. \nonumber \\ & \qquad \qquad \qquad \qquad \qquad \qquad \qquad \qquad \qquad \quad \ \ \left. \mathbf{x}\textrm{ is a $\beta$-rotation path}
\vphantom{\alcc\left(\armh \cap \bigcap_{0\le t\le \beta} L_t(\mathbf{x}(t))
\cap \Big(\mathbf{x}(\beta)+\rotmat{\beta}(\armv) \Big)\right)}
\right\}
\nonumber \\ &\le
\sup \left\{ \alcc\left(H \cap \bigcap_{j=1}^k L_{\alpha_j}(\mathbf{x}(\alpha_j))
\right) \,:\, \mathbf{x}\textrm{ is a $\beta$-rotation path}
\right\}
\nonumber \\ &=
\sup \left\{ \alcc\left(H \cap \bigcap_{j=1}^k L_{\alpha_j}(\mathbf{x}_j)
\right) \,:\, \mathbf{x}_1,\ldots,\mathbf{x}_k\in\R^2 
\right\}
= G_{\boldsymbol{\alpha}}.
\end{align*}
This proves claim (i) of the Proposition. 
If one further assumes that $\alpha_k\le \beta_0$, \eqref{eq:sofa-fg-bound2} also follows immediately using \eqref{eq:sofaconst-beta0}, proving claim (ii).

The proof of claim (iii) follows a variant of the same argument used above; first, note that we may assume that $\beta_2<\pi/2$, since the case $\beta_2=\pi/2$ already follows from part (i) of the Proposition. Next, observe that
\begin{align*}
\sofaconstbeta(\beta) &= 
\sup \left\{ \alcc\left(\armh \cap \bigcap_{0\le t\le \beta} L_t(\mathbf{x}(t))
\cap \Big(\mathbf{x}(\beta)+\rotmat{\beta}(\armv) \Big)\right) \,:\,
\right. \nonumber \\ & \qquad \qquad \qquad \qquad \qquad \qquad \qquad \qquad \qquad \quad \ \ \left. \mathbf{x}\textrm{ is a $\beta$-rotation path}
\vphantom{\alcc\left(\armh \cap \bigcap_{0\le t\le \beta} L_t(\mathbf{x}(t))
\cap \Big(\mathbf{x}(\beta)+\rotmat{\beta}(\armv) \Big)\right)}
\right\}
\\
& \le
\sup_{\mathbf{x}_1,\ldots,\mathbf{x}_{k+1}\in\R^2} \left[ \alcc\left(H \cap \bigcap_{j=1}^k L_{\alpha_j}(\mathbf{x}_j)
\cap \Big(\mathbf{x}_{k+1}+\rotmat{\beta}(V) \Big)\right)  \right].
\\
& =
\sup_{\mathbf{y}_1,\ldots,\mathbf{y}_{k}\in\R^2} \left[ \alcc\left(H \cap \bigcap_{j=1}^k L_{\alpha_j}(\mathbf{y}_j)
\cap \rotmat{\beta}(V) \right)  \right],
\\\end{align*}
where the last equality follows by expressing $\mathbf{x}_{k+1}$ in the form $\mathbf{x}_{k+1}=a (1,0)+b (-\sin \beta,\cos \beta)$,
making the substitution $\mathbf{x}_j = \mathbf{y}_j+a(1,0)$ ($1\le j\le k$), and using the facts that $H+a(1,0)=H$ and $\rotmat{\beta}(V)+b(-\sin \beta,\cos\beta) = \rotmat{\beta}(V)$. Finally, as noted after the definition of $B(\beta_1,\beta_2)$, this set has the property that if
$\beta \in [\beta_1,\beta_2]$ then
$\rotmat{\beta}(V) \subset B(\beta_1,\beta_2)$. We therefore get for such $\beta$ that
\begin{align*}
\sofaconstbeta(\beta) & \le
\sup_{\mathbf{y}_1,\ldots,\mathbf{y}_{k}\in\R^2} \left[ \alcc\left(H \cap \bigcap_{j=1}^k L_{\alpha_j}(\mathbf{y}_j)
\cap B(\beta_1,\beta_2) \right)  \right] = G_{\boldsymbol{\alpha}}^{\beta_1,\beta_2},
\end{align*}
which finishes the proof.
\end{proof}

\paragraph{Example.} In the case of a vector $\boldsymbol{\alpha}=(\alpha)$ with a single angle $0<\alpha<\pi/2$, a simple calculation, which we omit, shows that 
\begin{equation}
\label{eq:g-alpha-explicit}
G_{(\alpha)} = \sec\alpha+\csc\alpha.
\end{equation}
Taking $\alpha=\pi/4$ and using Proposition~\ref{prop:sofa-fg-bounds}(ii), we get the result that $\sofaconst\le 2\sqrt{2}$, which is precisely Hammersley's upper bound for $\sofaconst$ mentioned in the introduction (indeed, this application of the Proposition is essentially Hammersley's proof rewritten in our notation). 

\bigskip
We conclude this section with a result that makes precise the notion that the optimization problems involved in the definition of $G_{\boldsymbol{\alpha}}^{\beta_1,\beta_2}$ are finite-dimensional approximations to the (infinite-dimensional) optimization problem that is the moving sofa problem.

\begin{thm}[Convergence to the moving sofa problem]
\label{thm:conv-moving-sofa}
    Let 
\begin{align*}
\boldsymbol{\alpha}(n,k) &= (\tfrac1n\tfrac\pi2,\tfrac2n\tfrac\pi2,\tfrac3n\tfrac\pi2,\ldots,\tfrac{n-k-1}{n}\tfrac\pi2)\text, \\ 
\gamma_1(n,k)&=\tfrac{n-k}{n}\tfrac\pi2, \\ 
\gamma_2(n,k)&=\tfrac{n-k+1}{n}\tfrac\pi2,
\end{align*}
    and let 
\begin{equation} \label{eq:def-wn} 
    W_n = \max_{k=1,\ldots,\lceil n/3\rceil} G_{\boldsymbol{\alpha}(n,k)}^{\gamma_1(n,k),\gamma_2(n,k)}\text.
\end{equation}
    Then $\lim_{n\to\infty} W_n = \sofaconst$.
\end{thm}

To motivate the statement and proof of the theorem, note that 
the shape that achieves the area $G_{\boldsymbol{\alpha}}^{\beta_1,\beta_2}$ in the corresponding optimization problem is not generally a moving sofa shape, since it is chosen to satisfy only a finite number of the constraints a true moving sofa shape needs to satisfy. However, the idea is that we can take a sequence of these optimal shapes from optimization problems with longer and longer sequences of angles that are increasingly tightly spaced, and, through a limiting procedure (essentially a compactness argument), construct a moving sofa shape whose area is no less than the limit of the areas of the optimal shapes. This will establish that $\sofaconst$ can be approached from above by the numbers $W_n$, which are defined in terms of the finite-dimensional optimization problems. (In particular, this implies the rather weak statement that $\sofaconst$ is a computable number.)

\begin{proof}[Proof of Theorem~\ref{thm:conv-moving-sofa}]
Start by noting that, by Proposition~\ref{prop:sofa-fg-bounds}(iii), $W_n\ge \sofaconstbeta(\beta)$ for all $\pi/3\le\beta\le\pi/2$, so, because of \eqref{eq:sofaconst-beta0}, $W_n \ge \sofaconst$ for all $n$, and therefore
$\liminf_{n\to\infty} W_n \ge \sofaconst$.
Thus, it is enough to prove that $\sofaconst \ge \limsup_{n\to\infty} W_n$. 

    For each $n\ge 6$, denote by $k^*_n$ the smallest value of $k$ for which the maximum in the definition $W_n$ is attained, let $\beta_n' = \gamma_1(n,k_n^*)$, 
    and let $\beta_n''=\gamma_2(n,k_n^*)$. 
    There is subsequence of $W_n$ converging to its limit superior. Furthermore, the values of $\beta_n'$ in this subsequence have an accumulation point.
    Therefore, let $n_m$ denote the indices of a subsequence of $W_n$ converging to its limit superior such that $\beta_n'$ (and therefore also $\beta_n''$) converges to some limiting angle $\beta \in [\pi/3,\pi/2]$.

Let 
$\mathbf{u}^{(n)}=\left(u_1^{(n)},\ldots,u_{2(n-{k^*_n}-1)}^{(n)}\right)$ denote a point in
    $\Omega_{\boldsymbol{\alpha}(n,k^*_n)}^{\beta'_n,\beta''_n}$
    where $g_{\boldsymbol{\alpha}(n,{k^*_n})}^{\beta_n',\beta_n''}$ attains its maximum value, which (it is immediate from the definitions) is equal to $G_{\boldsymbol{\alpha}(n,{k^*_n})}^{\beta_n',\beta_n''}=W_n$. Moreover, as more than one point may attain this value, let $\mathbf{u}^{(n)}$ be chosen to be minimal under the coordinatewise partial order with respect to this property.

    Let $P_n$ be a largest area connected component of 
    $$H\cap \bigcap_{j=1}^{n-{k^*_n}-1} L_{j\pi/2n}\left(u^{(n)}_{2j-1},u^{(n)}_{2j}\right)\cap B\left(\beta_n',\beta_n''\right).$$
    Again by the definitions, $\lambda(P_n)=W_n$.
    Note that the diameter of $P_n$ is bounded from above by a universal constant; this is easy to check and is left as an exercise.
    
    Now, of course $P_n$ is not a moving sofa shape, but we will show that it approximates one along a suitable subsequence. To this end, define functions $U^{(n)}:[0,\beta]\to \R$, $V^{(n)}:[0,\beta]\to \R$ by
\begin{align*}
U^{(n)}(t) &= \max_{(x,y)\in P_n} \left( x\cos t+y\sin t - 1 \right), \\
V^{(n)}(t) &= \max_{(x,y)\in P_n} \left( -x\sin t+y\cos t - 1 \right).
\end{align*}
We note that $U^{(n)}(t)$ and $V^{(n)}(t)$ have the following properties: first, they are Lipschitz-continuous with a uniform (independent of $n$) Lipschitz constant; see pp.\ 269--270 of Gerver's paper \cite{gerver} for the proof, which uses the fact that the diameters of $P_n$ are bounded.

Second, the fact that $P_n \subset H$ implies that 
$\displaystyle V^{(n)}(0) = \max_{(x,y)\in P_n} (y-1) \le 0$. This in turn implies that
\begin{align} 
P_n &\subseteq 
(-\infty,U^{(n)}(0)+1] \times [0,V^{(n)}(0)+1]
\nonumber \\ & \subseteq
(-\infty,U^{(n)}(0)+1] \times [V^{(n)}(0),V^{(n)}(0)+1] \nonumber \\&=
(U^{(n)}(0),V^{(n)}(0)) + \armh.
\label{eq:third-observation} 
\end{align}

Third, we have
$$ \left(U^{(n)}(j \pi/2n), V^{(n)}(j \pi/2n) \right) = \left(u^{(n)}_{2j-1}, u^{(n)}_{2j} \right) $$
for all $j=1,\ldots,n-{k^*_n}$; that is, $U^{(n)}(t)$ and $V^{(n)}(t)$ continuously interpolate the odd and even (respectively) coordinates of the vector $\mathbf{u}^{(n)}$. Indeed, the relation $P_n \subset L_{j\pi/2n}\left(u^{(n)}_{2j-1},u^{(n)}_{2j}\right)$ implies trivially that 
$$U^{(n)}(j \pi/2n) \le u^{(n)}_{2j-1} \ \ \textrm{ and }\ \ 
V^{(n)}(j \pi/2n) \le u^{(n)}_{2j}.$$
However, if we had a strict inequality $U^{(n)}(j \pi/2n) < u^{(n)}_{2j-1}$ (respectively, $V^{(n)}(j \pi/2n) < u^{(n)}_{2j}$), that would imply that 
replacing $L_{j\pi/2n}\left(u^{(n)}_{2j-1},u^{(n)}_{2j}\right)$ by $L_{j\pi/2n}\left(u^{(n)}_{2j-1}-\epsilon,u^{(n)}_{2j}\right)$ (respectively, $L_{j\pi/2n}\left(u^{(n)}_{2j-1},u^{(n)}_{2j}-\epsilon\right)$) for some small positive $\epsilon$ in the definition of $P_n$ would not decrease the area, in contradiction to the maximality property defining $\mathbf{u}^{(n)}$.

We now define a smoothed version $T_n$ of the polygon $P_n$ by letting
 \begin{align*}
    T_n =& P_n\cap\left(\armh+(U^{(n)}(0),V^{(n)}(0))\right)
    \cap\bigcap_{0\le t\le\beta} L_t(U^{(n)}(t),V^{(n)}(t))\\
    &\cap    \left( \rotmat{\beta}\Big(\armv\Big)+(U^{(n)}(\beta),V^{(n)}(\beta)) \right)
    \\ =&
    P_n\cap\bigcap_{0\le t\le\beta} L_t(U^{(n)}(t),V^{(n)}(t))\cap    \left( \rotmat{\beta}\Big(\armv\Big)+(U^{(n)}(\beta),V^{(n)}(\beta)) \right),
\end{align*}
where the first equality sign is a definition, and the second equality follows from \eqref{eq:third-observation}.

Recall that the Hausdorff distance between compact sets $A$ and $B$ is the infimum of all $\epsilon>0$
such that for all $\mathbf{a}\in A$, there is $\mathbf{b}\in B$ with $\|\mathbf{a}-\mathbf{b}\|\le\epsilon$
and for all $\mathbf{b}\in B$, there is $\mathbf{a}\in A$ with $\|\mathbf{a}-\mathbf{b}\|\le\epsilon$.
It is known that the area of a compact set is upper semicontinuous with respect to the Hausdorff distance.
That is, if $A_n\to A$ under the Hausdorff distance for compact sets $A_n$, then $\lambda(A)\ge\limsup_{n\to\infty}\lambda(A_n)$; see \cite[Theorem~12.3.6]{schneider-weil}.

We claim that the Hausdorff distance between the sets $T_{n_m}$ and $P_{n_m}$ goes to $0$ as $n\to\infty$.
To see this, let $(x,y)\in P_{n_m}\setminus T_{n_m}$.
From the fact that $(x,y)\in P_{n_m}$ we have that
\begin{equation}
\label{eq:ineqPnm1}
\begin{aligned}
    x\cos t + y\sin t &\ge U^{(n_m)}(t) \  \textrm{ or } \ 
    -x\sin t + y\cos t \ge V^{(n_m)}(t) 
\end{aligned}
\end{equation}
for all $t = j\pi/2n_m$, where $j=1,\ldots,n_m-k_{n_m}^*-1$. Moreover, we have $y>V^{(n_m)}(0)$ and 
\begin{equation}
    \label{eq:ineqPnm2}
    x\cos t + y\sin t \ge U^{(n_m)}(t)
\end{equation}
for $t = \beta'_{n_m} = (n_m - k_{n_m}*)\pi/2n_m$.
We want to show that there exists $\delta_m\to 0$ such that
\begin{equation}
\begin{aligned}
\label{eq:ineqTnm1}
    x\cos t' + (y+\delta_m)\sin t' &\ge U^{(n_m)}(t') \  \textrm{ or } \\ 
    -x\sin t' + (y+\delta_m)\cos t' &\ge V^{(n_m)}(t').
\end{aligned}
\end{equation}
for all $0<t'<\beta$ and 
\begin{equation}
\label{eq:ineqTnm2}
    x\cos \beta + (y+\delta_m)\sin \beta \ge U^{(n_m)}(\beta)\text.
\end{equation}
We claim that $\delta_m = C((1/n_m) + |\beta-\beta'_{n_m}|)^{1/2}$ suffices, where $C$ is some constant.
First, if $\beta<\pi/2$, then for $t'\in[(1/n_m)^{1/2},\beta]$, we have \eqref{eq:ineqTnm1} 
from the uniform Lipschitz continuity of $U^{(n_m)}(t')$, $V^{(n_m)}(t')$, and the other terms as functions of $t'$
(recall $x$ and $y$ are uniformly bounded), from the fact that we have \eqref{eq:ineqPnm1} for some $t$ with
$|t-t'|<((1/n_m) + |\beta-\beta'_{n_m}|)$, and from $|\sin t'|, |\cos t'| > \tfrac12((1/n_m) + |\beta-\beta'_{n_m}|)^{1/2}$.
For $t'<(1/n_m)^{1/2}$, the fact that $y>V^{(n_m)}(0)$ and Lipschitz continuity suffice to give the second clause of \eqref{eq:ineqTnm1}.
Finally, \eqref{eq:ineqTnm2} is satisfied due to Lipschitz continuity and the inequality \eqref{eq:ineqPnm2}.
The case of $\beta=\pi/2$ can be worked out similarly.
Therefore, for every $(x,y)\in P_{n_m}\setminus T_{n_m}$, we can construct a point $(x,y')\in T_{n_m}$, with $|y'-y|\le \delta_m$ with $\delta_m\to0$.

We now use the fact that the vector-valued function $(U^{(n_m)}(t),V^{(n_m)}(t))$ is uniformly Lipschitz to conclude using the Arzel\`a-Ascoli theorem that
it has a subsequence (which we still denote by $n_m$, to avoid clutter) such that the ``anchored'' version of the function $(U^{(n_m)}(t),V^{(n_m)}(t))$, namely 
$$(U^{(n_m)}(t)-U^{(n_m)}(0), V^{(n_m)}(t)-V^{(n_m)}(0))$$ 
converges in the supremum norm to some limiting function
    $\mathbf{x}:[0,\beta]\to\R^2$, with the same Lipschitz constant, which satisfies $\mathbf{x}(0)=(0,0)$; that is, the limiting function $\mathbf{x}$ is a $\beta$-rotation path. Now let
    \begin{equation*}
	\begin{aligned}
    T_\infty =& \armh
    \cap\bigcap_{0\le t\le\beta} L_t(\mathbf{x}(t)) 
    \cap    \left(\rotmat{\beta}\left(\armv\right)+\mathbf{x}(\beta)\right).
	\end{aligned}
    \end{equation*}
    Since the Hausdorff distances of $P_{n_m}$ to $T_{n_m}$ and of $T_{n_m}-(U^{(n)}(0),V^{(n)}(0))$ to $T_\infty\cap \left(T_{n_m}-(U^{(n)}(0),U^{(n)}(0))\right)$ both approach zero as $m\to\infty$, 
    we have that the largest connected component of $T_\infty$ has an area at least as large as $\lim_{m\to\infty}  \lambda(P_{n_m}) = \lim\sup_{n\to\infty} W_n$.
    On the other hand, $T_\infty$ is of the form \eqref{eq:sx-intersections} for a $\beta$-rotation path $\mathbf{x}(t)$, so, by \eqref{eq:sofaconst-characterization}, its area is bounded from above by $\sofaconst$. We have shown that $\lim\sup_{n\to\infty} W_n\le\sofaconst$, and the proof is finished.
\end{proof}

We remark that in view of Theorem~\ref{thm:angle-bound}, it is easy to see that Theorem~\ref{thm:conv-moving-sofa} remains true if we replace the range $1\le k\le \lceil n/3 \rceil$ of values of $k$ in \eqref{eq:def-wn} with the smaller (and therefore computationally more efficient) range $1\le k\le \lceil n/9 \rceil$.

\section{An algorithmic proof scheme for moving sofa area bounds}

\label{sec:algorithm}

The theoretical framework we developed in the previous section reduces the problem of deriving upper bounds for $\sofaconst$ to that of proving upper bounds for the function $G_{\boldsymbol{\alpha}}^{\beta_1,\beta_2}$. Since this function is defined in terms of solutions to a family of optimization problems in finite-dimensional spaces, this is already an important conceptual advance. However, from a practical standpoint it remains to develop and implement a practical, efficient algorithm for solving optimization problems in this class. Our goal in this section is to present such an algorithm and establish its correctness.

Our computational strategy is a variant of the \textbf{geometric branch and bound} optimization technique \cite{ratscheck}.
Recall from Lemma~\ref{lem:supmax} that the maximum of the function $g_{\boldsymbol{\alpha}}^{\beta_1,\beta_2}$ is attained in a box (a Cartesian product of intervals) $\Omega_{\boldsymbol{\alpha}}^{\beta_1,\beta_2} \subset \R^{2k}$.
Our strategy is to break up $\Omega_{\boldsymbol{\alpha}}^{\beta_1,\beta_2}$ into sub-boxes.
On each box $E$ being considered, we will compute a quantity $\Gamma_{\boldsymbol{\alpha}}^{\beta_1,\beta_2}(E)$, which is an upper bound for $g_{\boldsymbol{\alpha}}^{\beta_1,\beta_2}(\mathbf{u})$ that holds uniformly for all $\mathbf{u} \in E$. In many cases this bound will not be an effective one, in the sense that there is a possibility that the box contains the point maximizing $g_{\boldsymbol{\alpha}}^{\beta_1,\beta_2}$; in such a case the box will be subdivided into two further boxes $E_1$ and $E_2$, which will be inserted into a queue to be considered later. Other boxes lead to effective bounds (that is, bounds that are smaller than a number already established as a lower bound for $G_{\boldsymbol{\alpha}}^{\beta_1,\beta_2}$) and need not be considered further. By organizing the computation efficiently, practical bounds can be established in a reasonable time, at least for small values of~$k$.

To make the idea precise, we introduce a few more definitions. Given two intervals $I=[a,b], J=[c,d] \subseteq \R$ and $\alpha \in[0,\pi/2]$, define
$$
\widehat{L}_\alpha(I,J) = \bigcup_{u \in I, v \in J} L_\alpha(u,v).
$$
Note that $\widehat{L}_\alpha(I,J)$ can also be expressed as a Minkowski sum of $L_\alpha(0,0)$ with the rotated rectangle $\rotmat{\alpha}(I\times J)$; in particular, it belongs to the class of sets known as \textbf{Nef polygons}, which are defined as planar sets that can be obtained from a finite set of half-planes by applying set intersection and complementation operations (see the Appendix for further discussion of the relevance of this fact to our software implementation of the algorithm).
Now, for a box $E=I_1\times \ldots \times I_{2k} \subset \R^2$, define
\begin{equation}\label{eq:defcalF}
\Gamma_{\boldsymbol{\alpha}}^{\beta_1,\beta_2}(E) = 
\alcc\left(
H \cap \bigcap_{j=1}^k 
\widehat{L}_{\alpha_j}(I_{2j-1},I_{2j})
\cap B(\beta_1,\beta_2)
\right)\text.
\end{equation}
Thus, by the definitions we have trivially that
\begin{equation} \label{eq:upperbound-trivially}
\sup_{\mathbf{u}\in E} g_{\boldsymbol{\alpha}}(\mathbf{u})
\le \Gamma_{\boldsymbol{\alpha}}^{\beta_1,\beta_2}(E).
\end{equation}
Next, given a box $E=I_1\times\ldots\times I_{2k}$ where $I_j=[a_j,b_j]$, let
$$
P_{\textrm{mid}}(E) = \left( \frac{a_1+b_1}{2}, \ldots, \frac{a_{2k}+b_{2k}}{2}\right)
$$
denote its midpoint.
We also assume that some rule is given to associate with each box $E$ a coordinate $i=\operatorname{ind}(E) \in \{1,\ldots,2k\}$, called the \textbf{splitting index} of $E$. This index will be used by the algorithm to split $E$ into two sub-boxes, which we denote by $\operatorname{split}_{i,1}(E)$ and $\operatorname{split}_{i,2}(E)$, and which are defined as
\begin{align*}
\operatorname{split}_{i,1}(E) &= I_1\times \ldots \times I_{i-1} \times \left[a_i,\tfrac12(a_i+b_i)\right] \times I_{i+1}\times \ldots I_{2k}, \\
\operatorname{split}_{i,2}(E) &= I_1\times \ldots \times I_{i-1} \times \left[\tfrac12(a_i+b_i),b_i\right] \times I_{i+1}\times \ldots I_{2k}.
\end{align*}
We assume that the mapping $E\mapsto \operatorname{ind}(E)$ has the property that, if the mapping $E \mapsto \operatorname{split}_{\operatorname{ind}(E),j}(E)$ is applied iteratively, with arbitrary choices of $j\in\{1,2\}$ at each step and starting from some initial value of $E$, the resulting sequence of splitting indices $i_1,i_2,\ldots$ contains each possible coordinate infinitely many times. A mapping satisfying this assumption is referred to as a \textbf{splitting rule}.

The algorithm is based on the standard data structure of a \textbf{priority queue} \cite{CLRS} used to hold boxes that are still under consideration. Recall that in a priority queue, each element of the queue is associated with a numerical value called its priority, and that the queue realizes operations of pushing a new element into the queue with a given priority, and popping the highest priority element from the queue. In our application, the priority of each box $E$ will be set to a value denoted $\Pi(E)$, where the mapping $E\mapsto \Pi(E)$ is given and is assumed to satisfy
\begin{equation} \label{eq:priority-map-condition}
\Pi(E) \ge \Gamma_{\boldsymbol{\alpha}}^{\beta_1,\beta_2}(E).
\end{equation}
Aside from this requirement, the precise choice of mapping is an implementation decision.
(A key point here is that setting $\Pi(E)$ \emph{equal} to $\Gamma_{\boldsymbol{\alpha}}^{\beta_1,\beta_2}(E)$ is conceptually the simplest choice, but from the practical point of view of minimizing programming complexity and running time it may not be optimal; see the Appendix for further discussion of this point.) Note that, since boxes are popped from the queue to be inspected by the algorithm in decreasing order of their priority, this ensures that the algorithm pursues successive improvements to the upper bound it obtains in a greedy fashion.

The algorithm also computes a lower bound on $G_{\boldsymbol{\alpha}}^{\beta_1,\beta_2}$ by evaluating $g_{\boldsymbol{\alpha}}^{\beta_1,\beta_2}(\mathbf{u})$
at the midpoint of every box it processes and keeping track of the largest value observed. This lower bound is used to discard boxes in which it is impossible for the maximum to lie. The variable keeping track of the lower bound is initialized to some number $\ell_0$ known to be a lower bound for $G_{\boldsymbol{\alpha}}^{\beta_1,\beta_2}$. In our software implementation we used the value
$$
\ell_0 = \begin{cases} 0 & \textrm{if }\beta_2 < \pi/2, \\ 11/5 & \textrm{if }\beta_2 = \pi/2,
\end{cases}
$$
this being a valid choice thanks to the fact that (by Proposition~\ref{prop:sofa-fg-bounds}(iii)) $G_{\boldsymbol{\alpha}}^{\beta_1,\pi/2} \ge \sofaconstbeta(\pi/2) \ge \gerverconst = 2.2195\ldots > 2.2=11/5$. Note that simply setting $\ell_0=0$ in all cases would also result in a valid algorithm, but would result in a slight waste of computation time compared to the definition above.

With this setup, we can now describe the algorithm, given in pseudocode in Listing~\ref{alg:branchandbound}. 

The next proposition is key to proving the algorithm's correctness.

\begin{listing}
\begin{mdframed}[backgroundcolor=codebgcolor] 
\begin{algorithmic}
\State $\varname{box\_queue} \gets $ an empty priority queue of boxes
\State $\varname{initial\_box} \gets $ box representing $\Omega_{\boldsymbol{\alpha}}^{\beta_1,\beta_2}$, computed according to the \State \phantom{$\varname{initial\_box} \gets $} function of $\boldsymbol{\alpha}, \beta_1, \beta_2$ described in Lemma~\ref{lem:supmax}
\smallskip
\State \keyword{push} $\varname{initial\_box}$ into $\varname{box\_queue}$ with priority $\Pi(\varname{initial\_box})$
\smallskip
\State $\varname{best\_lower\_bound\_so\_far} \gets $ the initial lower bound $\ell_0$
\medskip
\While{true}
\medskip
\State \keyword{pop} highest priority element of $\varname{box\_queue}$ into $\varname{current\_box}$
\State $\varname{current\_box\_lower\_bound} \gets g_{\boldsymbol{\alpha}}^{\beta_1,\beta_2}(P_{\textrm{mid}}(\varname{current\_box}))$
\State $\varname{best\_upper\_bound\_so\_far} \gets \Pi(\varname{current\_box})$
\medskip
\If{$\varname{current\_box\_lower\_bound} > \varname{best\_lower\_bound\_so\_far}$}
\State $\varname{best\_lower\_bound\_so\_far} \gets \varname{current\_box\_lower\_bound}$
\EndIf
\medskip
\State $\varname{i} \gets \operatorname{ind}(\varname{current\_box})$
\smallskip
\smallskip
\For{$\varname{j}=1,2$}
\State $\varname{new\_box} \gets \operatorname{split}_{\varname{i},\varname{j}}(\varname{current\_box})$ 
\If{$\Pi(\varname{new\_box}) \ge \varname{best\_lower\_bound\_so\_far}$
}
\State \keyword{push} $\varname{new\_box}$ into $\varname{box\_queue}$ with priority
$\Pi(\varname{new\_box})$
\EndIf
\EndFor
\medskip
\State \textbf{Reporting point:} print the values of $\varname{best\_upper\_bound\_so\_far}$
\State \phantom{\textbf{Reporting point:}} and $\varname{best\_lower\_bound\_so\_far}$
\medskip
\EndWhile
\end{algorithmic}
\end{mdframed}
\caption{The algorithm for computing bounds for $G_{\boldsymbol{\alpha}}^{\beta_1,\beta_2}$.}
\label{alg:branchandbound}
\end{listing}

\begin{prop}\label{prop:lead-box}
    Any box $\tilde E$ which is the highest priority box in the queue $\varname{box\_queue}$ at some step satisfies
\begin{equation} \label{eq:lead-box}
\sup\{g_{\boldsymbol{\alpha}}^{\beta_1,\beta_2}(\mathbf{u}):\mathbf{u}\in\tilde E\} \le G_{\boldsymbol{\alpha}}^{\beta_1,\beta_2} \le \Pi(\tilde E)\text.
\end{equation}
\end{prop}

\begin{proof}
First, the lower inequality holds for \emph{all} boxes in the queue, simply because the value $g_{\boldsymbol{\alpha}}^{\beta_1,\beta_2}(\mathbf{u})$ for any $\mathbf{u}\in\Omega_{\boldsymbol{\alpha}}^{\beta_1,\beta_2}$ is a lower bound on its maximum over all $\mathbf{u}\in\Omega_{\boldsymbol{\alpha}}^{\beta_1,\beta_2}$.

Next, let $Q_n$ denote the collection of boxes in the priority queue after $n$ iterations of the \texttt{while} loop (with $Q_0$ being the initialized queue containing the single box $\Omega_{\boldsymbol{\alpha}}^{\beta_1,\beta_2}$), and let $D_n$ denote the collection of boxes that were discarded (not pushed into the priority queue during the execution of the \texttt{if} clause inside the \texttt{for} loop) during the first $n$ iterations of the \texttt{while} loop. Then we first note that for all $n$, the relation
\begin{equation} \label{eq:confspace-decom}
\Omega_{\boldsymbol{\alpha}}^{\beta_1,\beta_2} = \bigcup_{E \in Q_n\cup D_n} E
\end{equation}
holds.
Indeed, this is easily proved by induction on $n$: if we denote by $X$ the highest priority element in $Q_n$, then during the $(n+1)$th iteration of the \texttt{while} loop, $X$ is subdivided into two boxes $X=X_1\cup X_2$, and each of $X_1, X_2$ is either pushed into the priority queue (i.e., becomes an element of $Q_{n+1}$) or discarded (i.e., becomes an element of $D_{n+1}$), so we have that $\Omega_{\boldsymbol{\alpha}}^{\beta_1,\beta_2} = \bigcup_{E \in Q_{n+1}\cup D_{n+1}} E$, completing the inductive step.

Second, note that for any box $X\in D_n$, since $X$ was discarded during the $k$th iteration of the \texttt{while} loop for some $1\le k\le n$, we have that 
$\Pi(X)$
is smaller than the value of $\varname{best\_lower\_bound\_so\_far}$ during that iteration. But $\varname{best\_lower\_bound\_so\_far}$ is always assigned a value of the form $g_{\boldsymbol{\alpha}}^{\beta_1,\beta_2}(\mathbf{u})$ for some $\mathbf{u}\in R$ and is therefore bounded from above by $G_{\boldsymbol{\alpha}}^{\beta_1,\beta_2}$,
so we have established that
\begin{equation} \label{eq:xupperbound-le}
    \Pi(X) < G_{\boldsymbol{\alpha}}^{\beta_1,\beta_2} \qquad (X \in D_n).
\end{equation}
The relations \eqref{eq:supmax}, \eqref{eq:upperbound-trivially}, \eqref{eq:priority-map-condition}, \eqref{eq:confspace-decom}, and \eqref{eq:xupperbound-le} now imply that
\begin{align*}
G_{\boldsymbol{\alpha}}^{\beta_1,\beta_2} &= \max \left\{ g_{\boldsymbol{\alpha}}^{\beta_1,\beta_2}(\mathbf{u})
\,:\,
\mathbf{u}\in \Omega_{\boldsymbol{\alpha}}^{\beta_1,\beta_2} \right\} \\ &=
\max_{E \in Q_n\cup D_n} \left(
\sup \left\{ g_{\boldsymbol{\alpha}}^{\beta_1,\beta_2}(\mathbf{u})
\,:\,
\mathbf{u}\in E \right\}
\right)
\le
\max_{E \in Q_n\cup D_n} \Gamma_{\boldsymbol{\alpha}}^{\beta_1,\beta_2}(E)
\\ &\le \max_{E \in Q_n\cup D_n} \Pi(E)
=
\max_{E \in Q_n} \Pi(E).
\end{align*}
Finally, $\max_{E \in Q_n} \Pi(E) = \Pi(\tilde E)$, since $\tilde E$ was assumed to be the box with highest priority among the elements of $Q_n$, so we get the upper inequality in \eqref{eq:lead-box}, which finishes the proof.
\end{proof}

We immediately have the correctness of the algorithm as a corollary:

\begin{thm}[Correctness of the algorithm]
    Any \,\!\!\! value of the  variable \ $\varname{best\_upper\_bound\_so\_far}$ reported by the algorithm is an upper bound for $G_{\boldsymbol{\alpha}}^{\beta_1,\beta_2}$.
\end{thm}

Note that the correctness of the algorithm is not dependent on the assumption we made on the splitting index mapping $E\mapsto \operatorname{ind}(E)$ being a splitting rule. The importance of that assumption is explained by the following result, which also explains one sense in which assuming an equality in \eqref{eq:priority-map-condition} rather than an inequality provides a benefit (of a theoretical nature at least).

\begin{thm}[Asymptotic sharpness of the algorithm]
\label{thm:asym-sharpness}
Assume that the priority mapping $E\mapsto \Pi(E)$ is taken to be
\begin{equation} \label{eq:priority-map-sharpness}
\Pi(E) = \Gamma_{\boldsymbol{\alpha}}^{\beta_1,\beta_2}(E).
\end{equation}
Then the upper and lower bounds output by the algorithm both converge to $G_{\boldsymbol{\alpha}}^{\beta_1,\beta_2}$.
\end{thm}
\begin{proof}
    As one may easily check, the upper bound used in the calculation under the assumption \eqref{eq:priority-map-sharpness}, $\Gamma_{\boldsymbol{\alpha}}^{\beta_1,\beta_2}(E)$, approaches the actual supremum of $g_{\boldsymbol{\alpha}}^{\beta_1,\beta_2}(\mathbf{u})$ over $E$ as the diameter of $E$ approaches zero. 
    That is $|\Gamma_{\boldsymbol{\alpha}}^{\beta_1,\beta_2}(E) - \sup\{g_{\boldsymbol{\alpha}}^{\beta_1,\beta_2}(\mathbf{u}):\mathbf{u}\in E\}|$ is bounded by
    a function of the diameter of $E$ that approaches zero when the diameter approaches zero.
    The same is true of the variation in each box, $|\sup\{g_{\boldsymbol{\alpha}}^{\beta_1,\beta_2}(\mathbf{u}):\mathbf{u}\in E\} - \inf\{g_{\boldsymbol{\alpha}}^{\beta_1,\beta_2}(\mathbf{u}):\mathbf{u}\in E\}|$.
    
    When using a valid splitting rule, the diameter of the leading box approaches zero as $n$ approaches infinity, and Proposition~\ref{prop:lead-box} completes the proof.
\end{proof}

As with the case of the choice of priority mapping and the value of the initial lower bound $\ell_0$, the specific choice of splitting rule to use is an implementation decision, and different choices can lead to algorithms with different performance. A simple choice we tried was to use the index of the coordinate with the largest variation within $E$ (i.e., the ``longest dimension'' of $E$). Another choice, which we found gives superior performance and is the rule currently used in our software implementation \texttt{SofaBounds}, is to let the splitting index be the value of $i$ maximizing
$\lambda(D_i\cap S(E))$, where $S(E)$ is the argument of $\alcc$ in \eqref{eq:defcalF},
and 
\begin{align*}
D_i = \begin{cases}
\displaystyle \bigcup_{u\in {I_{2j-1}}} \widehat{L}_{\alpha_j}(u,I_{2j}) \setminus \bigcap_{u\in {I_{2j-1}}} \widehat{L}_{\alpha_j}(u,I_{2j})
& \textrm{if }i=2j-1, \\[14pt]
\displaystyle 
\bigcup_{u\in {I_{2j}}} \widehat{L}_{\alpha_j}(I_{2j-1},u) \setminus \bigcap_{u\in {I_{2j}}} \widehat{L}_{\alpha_j}(I_{2j-1},u) & \textrm{if }i=2j.
\end{cases}
\end{align*}

\section{Explicit numerical bounds}

\label{sec:numerical}

We report the following explicit numerical bounds obtained by our algorithm, which we will then use to prove Theorems~\ref{thm:new-upperbound} and~\ref{thm:angle-bound}.

\begin{thm}
\label{thm:explicit-bounds}
Define angles
\begin{align*}
\alpha_1 &= \sin^{-1}\tfrac{7}{25} \approx 16.26^\circ, \\
\alpha_2 &= \sin^{-1}\tfrac{33}{65} \approx 30.51^\circ, \\
\alpha_3 &= \sin^{-1}\tfrac{119}{169} \approx 44.76^\circ, \\
\alpha_4 &= \sin^{-1} \tfrac{56}{65} = \pi/2-\alpha_2 \approx 59.59^\circ, \\
\alpha_5 &= \sin^{-1}\tfrac{24}{25} = \pi/2-\alpha_1 \approx 73.74^\circ, \\
\alpha_6 &= \sin^{-1} \tfrac{60}{61} \approx 79.61^\circ, \\
\alpha_7 &= \sin^{-1} \tfrac{84}{85} \approx 81.2^\circ.
\end{align*}
Then we have the inequalities
\begin{align}
G_{(\alpha_1,\alpha_2,\alpha_3,\alpha_4,\alpha_5)} &\le 
\bestupperbound, \label{eq:numerical-bound1} \\[5pt]
G_{(\alpha_1,\alpha_2,\alpha_3)}^{\alpha_4,\alpha_5} &\le 
2.21,
\label{eq:numerical-bound2}
\\
G_{(\alpha_1,\alpha_2,\alpha_3)}^{\alpha_5,\alpha_6} &\le 2.21,
\label{eq:numerical-bound3}
\\
G_{(\alpha_1,\alpha_2,\alpha_3,\alpha_4)}^{\alpha_6,\alpha_7} &\le 2.21.
\label{eq:numerical-bound4}
\end{align}
\end{thm}

\begin{proof} Each of the inequalities \eqref{eq:numerical-bound1}--\eqref{eq:numerical-bound4} is certified as correct using the \texttt{SofaBounds} software package by invoking the \texttt{run} command from the command line interface after loading the appropriate parameters. For \eqref{eq:numerical-bound1}, the parameters can be loaded from the saved profile file \texttt{thm9-bound1.txt} included with the package (see the Appendix below for an illustration of the syntax for loading the file and running the computation). Similarly, the inequalities \eqref{eq:numerical-bound2}, \eqref{eq:numerical-bound3}, \eqref{eq:numerical-bound4} are obtained by running the software with the profile files \texttt{thm9-bound2.txt}, \texttt{thm9-bound3.txt}, and \texttt{thm9-bound4.txt}, respectively. Table~\ref{table:benchmarking} in the Appendix shows benchmarking results with running times for each of the computations.
\end{proof}

\begin{proof}[Proof of Theorem~\ref{thm:new-upperbound}]
For angles $0\le \beta_1<\beta_2 \le \pi/2$ denote 
$$M(\beta_1,\beta_2) = 
\sup_{\beta_1\le \beta \le \beta_2} \sofaconstbeta(\beta).$$
By \eqref{eq:sofaconst-beta0}, we have
\begin{equation} \label{eq:sofaconst-tworanges}
\sofaconst = 
M(\beta_0,\pi/2) =
\max\Big(
M(\beta_0,\alpha_5),
M(\alpha_5,\pi/2)
\Big).
\end{equation}
By Proposition~\ref{prop:sofa-fg-bounds}(iii), 
$M(\beta_0,\alpha_5)$
is bounded from above by $G_{(\alpha_1,\alpha_2,\alpha_3)}^{\alpha_4,\alpha_5}$, and by Proposition~\ref{prop:sofa-fg-bounds}(i), 
$M(\alpha_5,\pi/2)$
is bounded from above by $G_{(\alpha_1,\alpha_2,\alpha_3,\alpha_4,\alpha_5)}$. Thus, combining \eqref{eq:sofaconst-tworanges} with the numerical bounds \eqref{eq:numerical-bound1}--\eqref{eq:numerical-bound2} proves \eqref{eq:upperbound}.
\end{proof}

\begin{proof}[Proof of Theorem~\ref{thm:angle-bound}]
Using the same notation as in the proof of Theorem~\ref{thm:new-upperbound} above, we note that
$$
M(0,\alpha_7) 
= \max\Big(
M(0,\alpha_4), M(\alpha_4,\alpha_5), M(\alpha_5,\alpha_6), M(\alpha_6,\alpha_7)
\Big).
$$
Now, by Gerver's observation mentioned after the relation \eqref{eq:sofaconst-beta}, we have that $M(0,\alpha_4) \le \sec(\alpha_4) < \sec(\pi/3) = 2$. By Proposition~\ref{prop:sofa-fg-bounds}(iii) coupled with the numerical bounds \eqref{eq:numerical-bound2}--\eqref{eq:numerical-bound4}, the remaining three arguments
$M(\alpha_4,\alpha_5)$, $M(\alpha_5,\alpha_6)$, and $M(\alpha_6,\alpha_7)$ in the maximum are all bounded from above by $2.21$, so in particular we get that $M(0,\alpha_7)\le 2.21 < \gerverconst \approx 2.2195$. On the other hand, we have that
$$ \sofaconst = M(0,\pi/2) = \max\Big( M(0,\alpha_7), M(\alpha_7,\pi/2) \Big) \ge \gerverconst. $$
We conclude that $\sofaconst = M(\alpha_7,\pi/2)$ and that $\sofaconstbeta(\beta) \le 2.21 < \sofaconst$ for all $\beta<\alpha_7$. This proves that a moving sofa of maximal area has to undergo rotation by an angle of at least $\alpha_7$, as claimed.
\end{proof}

\section{Concluding remarks}

The results of this paper represent the first progress since Hammersley's 1968 paper \cite{hammersley} on deriving upper bounds for the area of a moving sofa shape. Our techniques also enable us to prove an improved lower bound on the angle of rotation a maximal area moving sofa shape must rotate through. Our improved upper bound of $\bestupperbound$ on the moving sofa constant comes much closer than Hammersley's bound to the best known lower bound $\gerverconst \approx 2.2195$ arising from Gerver's construction, but clearly there is still considerable room for improvement in narrowing the gap between the lower and upper bounds. In particular, some experimentation with the initial parameters used as input for the \texttt{SofaBounds} software should make it relatively easy to produce further (small) improvements to the value of the upper bound.

More ambitiously, our hope is that a refinement of our methods---in the form of theoretical improvements and/or speedups in the software implementation, for example using parallel computing techniques---may eventually be used to obtain an upper bound that comes very close to Gerver's bound, thereby providing supporting evidence to his conjecture that the shape he found is the solution to the moving sofa problem. Some supporting evidence of this type, albeit derived using a heuristic algorithm, was reported in a recent paper by Gibbs \cite{gibbs}. Alternatively, a failure of our algorithm (or improved versions thereof) to approach Gerver's lower bound may provide clues that his conjecture may in fact be false.

Our methods should also generalize in a fairly straightforward manner to other variants of the moving sofa problem. In particular, in a recent paper \cite{romik}, one of us discovered a shape with a piecewise algebraic boundary that is a plausible candidate to being the solution to the so-called \textbf{ambidextrous moving sofa problem}, which asks for the largest shape that can be moved around a right-angled turn \textit{either to the left or to the right} in a hallway of width~1 (Fig.~\ref{fig:romik-sofa}(a)). The shape, shown in Fig.~\ref{fig:romik-sofa}(b), has an area given by the intriguing explicit constant

\begin{align*}
\romikconst &=\sqrt[3]{3+2 \sqrt{2}}+\sqrt[3]{3-2 \sqrt{2}}-1 
+\arctan\left[
\frac{1}{2} \left( \sqrt[3]{\sqrt{2}+1}- \sqrt[3]{\sqrt{2}-1}\, \right)
  \right]
\nonumber \\ & \qquad\qquad\qquad = 1.64495521\ldots 
\end{align*}
As with the case of the original (non-ambidextrous) moving sofa problem, the constant $\romikconst$ provides a lower bound on the maximal area of an ambidextrous moving sofa shape; in the opposite direction, any upper bound for the original problem is also an upper bound for the ambidextrous variant of the problem, which establishes $\bestupperbound$ as a valid upper bound for that problem. Once again, the gap between the lower and upper bounds seems like an appealing opportunity for further work, so it would be interesting to extend the techniques of this paper to the setting of the ambidextrous moving sofa problem so as to obtain better upper bounds on the ``ambidextrous moving sofa constant.''

\begin{figure}
\begin{center}
\begin{tabular}{cc}
\scalebox{0.3}{\includegraphics{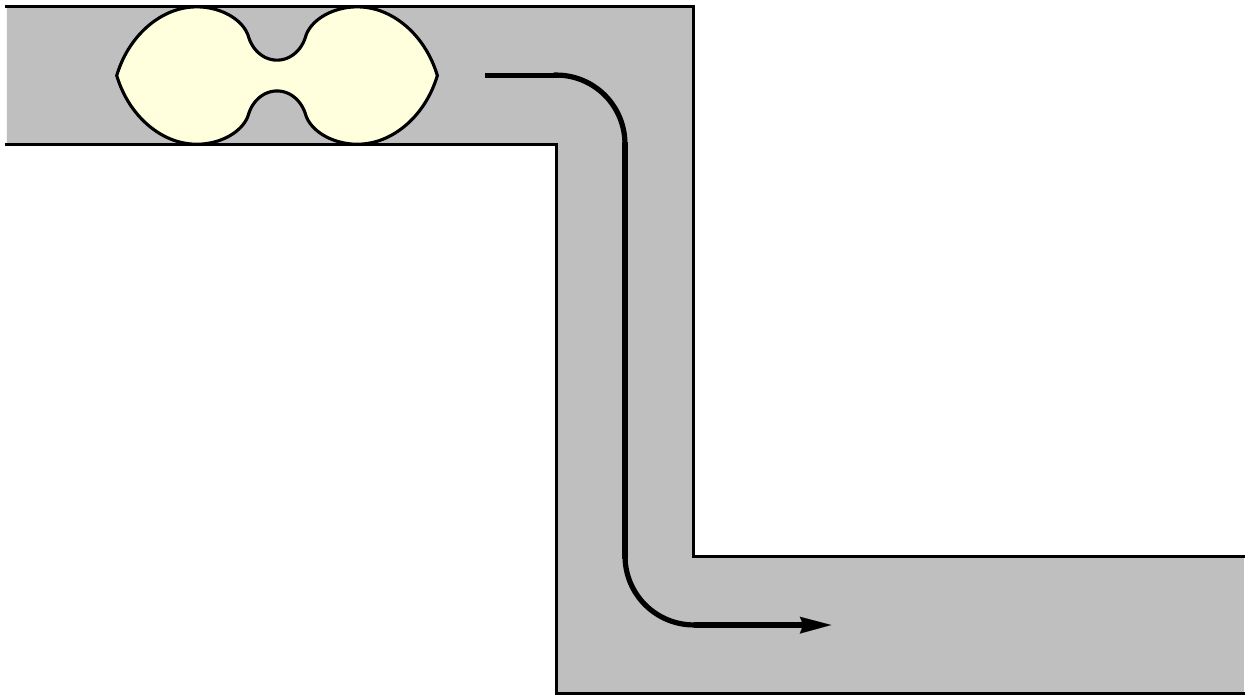}} & 
\raisebox{10pt}{\scalebox{0.5}{\includegraphics{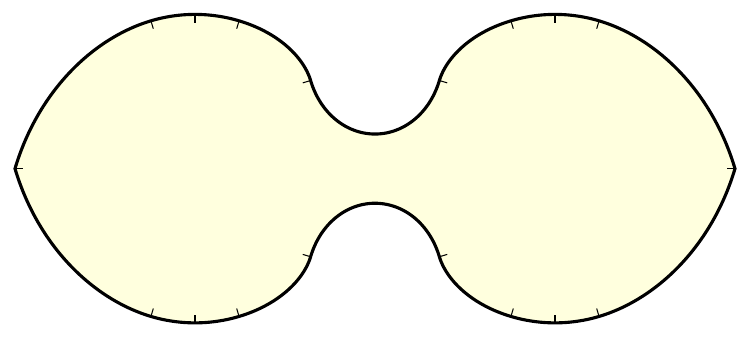}}}
\\[3pt]
(a) & (b)
\end{tabular}
\caption{(a) The ambidextrous moving sofa problem involves maximization of the area of moving sofa shapes that can navigate a hallway with right-angled turns going both ways, as shown in the figure; (b) a shape discovered by Romik \cite{romik} that was derived as a possible solution to the ambidextrous moving sofa problem. The boundary of the shape is a piecewise algebraic curve; the tick marks in the figure delineate the transition points between distinct parts of the boundary.}
\label{fig:romik-sofa}
\end{center}
\end{figure}

\section*{Appendix: The \texttt{SofaBounds} software}

We implemented the algorithm described in Section~\ref{sec:algorithm} in the software package \texttt{SofaBounds} we developed, which serves as a companion package to this paper and whose source code is available to download online \cite{sofabounds}. The package is a Unix command line tool written in \texttt{C++} and makes use of the open source computational geometry library \texttt{CGAL} \cite{cgal}. All computations are done in the exact rational arithmetic mode supported by \texttt{CGAL} to ensure that the bounds output by the algorithm are mathematically rigorous. For this reason, the software only works with angles $\gamma$ for which the vector $(\cos\gamma, \sin\gamma)$ has rational coordinates, i.e., is a rational point $(a/c,b/c)$ on the unit circle; clearly such angles are parametrized by Pythagorean triples $(a,b,c)$ such that $a^2+b^2=c^2$, and it is using such triples that the angles are entered into the program as input from the user. For example, to approximate an angle of $45$ degrees, we used the Pythagorean triple $(119,120,169)$, which corresponds to an angle of $\sin^{-1}(119/169) \approx 44.76^\circ$.

\texttt{SofaBounds} uses the \textbf{Nef polygon} geometric primitive implemented in \texttt{CGAL}. In Section~\ref{sec:algorithm} we mentioned that the set $\widehat{L}_\alpha(I,J)$ is a Nef polygon; consequently it is easy to see that all the planar sets manipulated by the algorithm belong to this family and can be readily calculated using elementary geometry and the features of \texttt{CGAL}'s Nef polygon sub-library \cite{seel}.

Our implementation uses the priority rule
$$
\Pi(E) = 
\lambda\left(
H \cap \bigcap_{j=1}^k 
\widehat{L}_{\alpha_j}(I_{2j-1},J_{2j})
\cap B(\beta_1,\beta_2)
\right)\text,
$$
i.e., we use the total area of the intersection as the priority instead of the area of the largest connected component as in \eqref{eq:defcalF}; this is slightly less ideal from a theoretical point of view, since Theorem~\ref{thm:asym-sharpness} does not apply, but simplified the programming and in practice probably results in better computational performance.

The software runs our algorithm on a single Unix thread, since the parts of the CGAL library we used are not thread-safe; note however that the nature of our algorithm lends itself fairly well to parallelization, so a multithreading or other parallelized implementation could yield a considerable speedup in performance, making it more practical to continue to improve the bounds in Theorems~\ref{thm:new-upperbound} and \ref{thm:angle-bound}.

To illustrate the use of the software, Listing~\ref{code-listing} shows a sample working session in which the upper bound $2.5$ is derived for $G_{\boldsymbol{\alpha}}$ with 
\begin{equation} \label{eq:angles-30-45-60}
\boldsymbol{\alpha}=\left(\sin^{-1}\frac{33}{65},\sin^{-1}\frac{119}{169}, \sin^{-1}\frac{56}{65}\right)
\approx (30.51^\circ, 44.76^\circ, 59.49^\circ).
\end{equation}

The numerical bounds \eqref{eq:numerical-bound1}--\eqref{eq:numerical-bound4} used in the proofs of Theorems~\ref{thm:new-upperbound} and~\ref{thm:angle-bound} were proved using \texttt{SofaBounds}, and required several weeks of computing time on a desktop computer. Table~\ref{table:benchmarking} shows some benchmarking information, which may be useful to anyone wishing to reproduce the computations or to improve upon our results.

\newcommand{\ignore}[1]{{}}
\begin{listing}
\begin{mdframed}[backgroundcolor=codebgcolor] 
\begin{alltt}
Users/user/SofaBounds$ \inputline{SofaBounds}
SofaBounds version 1.0

Type "help" for instructions.

> \inputline{load example-30-45-60.txt}
File 'example-30-45-60.txt' loaded successfully.
> \inputline{settings}

Number of corridors: 3

Slope 1:        33      56      65     (angle: 30.5102 deg)
Slope 2:       119     120     169     (angle: 44.7603 deg)
Slope 3:        56      33      65     (angle: 59.4898 deg)
Minimum final slope: 1 0 1	            (angle: 90 deg)
Maximum final slope: 1 0 1	            (angle: 90 deg)

Reporting progress every:		0.01 decrease in upper bound
> \inputline{run}
<iterations=0>
<iterations=1 | upper bound=3.754 | time=0:00:00> 
<iterations=7 | upper bound=3.488 | time=0:00:01> 
<iterations=9 | upper bound=3.438 | time=0:00:01> \vspace{7pt} 
\ignore{<iterations=13 | upper bound=3.428 | time=0:00:02> 
<iterations=14 | upper bound=3.416 | time=0:00:02> 
<iterations=16 | upper bound=3.405 | time=0:00:02> 
<iterations=18 | upper bound=3.397 | time=0:00:03> 
<iterations=19 | upper bound=3.380 | time=0:00:03> 
<iterations=24 | upper bound=3.360 | time=0:00:04>
<iterations=25 | upper bound=3.301 | time=0:00:04> 
<iterations=30 | upper bound=3.273 | time=0:00:05> 
<iterations=35 | upper bound=3.248 | time=0:00:06> 
<iterations=36 | upper bound=3.222 | time=0:00:07> 
<iterations=37 | upper bound=3.202 | time=0:00:07> 
<iterations=41 | upper bound=3.189 | time=0:00:08> 
<iterations=45 | upper bound=3.162 | time=0:00:09> 
<iterations=46 | upper bound=3.140 | time=0:00:09> 
<iterations=47 | upper bound=3.043 | time=0:00:09> 
<iterations=51 | upper bound=3.010 | time=0:00:10> 
<iterations=53 | upper bound=2.996 | time=0:00:11> 
<iterations=56 | upper bound=2.983 | time=0:00:11> 
<iterations=64 | upper bound=2.970 | time=0:00:14> 
<iterations=67 | upper bound=2.957 | time=0:00:14> 
<iterations=73 | upper bound=2.943 | time=0:00:16> 
<iterations=79 | upper bound=2.928 | time=0:00:18> 
<iterations=87 | upper bound=2.918 | time=0:00:20> 
<iterations=92 | upper bound=2.903 | time=0:00:21> 
<iterations=95 | upper bound=2.870 | time=0:00:22> 
<iterations=102 | upper bound=2.857 | time=0:00:24> 
<iterations=111 | upper bound=2.848 | time=0:00:27> 
<iterations=116 | upper bound=2.838 | time=0:00:28> 
<iterations=118 | upper bound=2.827 | time=0:00:29> 
<iterations=127 | upper bound=2.817 | time=0:00:31> 
<iterations=134 | upper bound=2.809 | time=0:00:34> 
<iterations=139 | upper bound=2.796 | time=0:00:35> 
<iterations=143 | upper bound=2.790 | time=0:00:36> 
<iterations=151 | upper bound=2.778 | time=0:00:38> 
<iterations=165 | upper bound=2.763 | time=0:00:42> 
<iterations=173 | upper bound=2.748 | time=0:00:44> 
<iterations=181 | upper bound=2.738 | time=0:00:47> 
<iterations=204 | upper bound=2.730 | time=0:00:53> 
<iterations=220 | upper bound=2.719 | time=0:00:58> 
<iterations=240 | upper bound=2.709 | time=0:01:03> 
<iterations=278 | upper bound=2.699 | time=0:01:14> 
<iterations=300 | upper bound=2.689 | time=0:01:21> 
<iterations=326 | upper bound=2.680 | time=0:01:28> 
<iterations=361 | upper bound=2.670 | time=0:01:38> 
<iterations=413 | upper bound=2.660 | time=0:01:53> 
<iterations=462 | upper bound=2.650 | time=0:02:07> 
<iterations=548 | upper bound=2.640 | time=0:02:31> 
<iterations=619 | upper bound=2.630 | time=0:02:52> 
<iterations=724 | upper bound=2.620 | time=0:03:22> 
<iterations=812 | upper bound=2.610 | time=0:03:48> 
<iterations=945 | upper bound=2.600 | time=0:04:27> 
<iterations=1100 | upper bound=2.590 | time=0:05:12> 
<iterations=1290 | upper bound=2.580 | time=0:06:07> 
<iterations=1513 | upper bound=2.570 | time=0:07:18> } \textnormal{[\textit{... 54 output lines deleted ...}]}\vspace{7pt} 
<iterations=1776 | upper bound=2.560 | time=0:08:43> 
<iterations=2188 | upper bound=2.550 | time=0:10:48> 
<iterations=2711 | upper bound=2.540 | time=0:13:23> 
<iterations=3510 | upper bound=2.530 | time=0:18:18> 
<iterations=4620 | upper bound=2.520 | time=0:24:54> 
<iterations=6250 | upper bound=2.510 | time=0:34:52> 
<iterations=8901 | upper bound=2.500 | time=0:50:45> 
\end{alltt}
\end{mdframed}
\caption{A sample working session of the \texttt{SofaBounds} software package proving an upper bound for $G_{\boldsymbol{\alpha}}$ with $\boldsymbol{\alpha}$ given by \eqref{eq:angles-30-45-60} . User commands are colored in \inputline{\textrm{blue}}.
The session loads parameters from a saved profile file \texttt{example-30-45-60.txt} (included with the source code download package) and rigorously certifies the number $2.5$ as an upper bound for 
$G_{\boldsymbol{\alpha}}$ (and therefore also for $\sofaconst$, by Proposition~\ref{prop:sofa-fg-bounds}(ii)) in about $50$ minutes of computation time on a laptop with a 1.3 GHz Intel Core M processor.
}
\label{code-listing}
\end{listing}

\begin{table}
\begin{center}
\begin{tabular}{|c|c|c|c|}
\hline
Bound & Saved profile file & Num.\ of iterations & Computation time \\
\hline
\eqref{eq:numerical-bound1} & \texttt{thm9-bound1.txt} &  7,724,162 & 480 hours \\
\eqref{eq:numerical-bound2} & \texttt{thm9-bound2.txt} &  \phantom{0,000,}917 & 2 minutes \\
\eqref{eq:numerical-bound3} & \texttt{thm9-bound3.txt} &  \phantom{,00}26,576 & 1:05 hours \\
\eqref{eq:numerical-bound4} & \texttt{thm9-bound4.txt} &  \phantom{0,}140,467 & 6:23 hours \\
\hline
\end{tabular}
\caption{Benchmarking results for the computations used in the proof of the bounds \eqref{eq:numerical-bound1}--\eqref{eq:numerical-bound4}. The computations for \eqref{eq:numerical-bound1} were performed on a 2.3 GHz Intel Xeon E5-2630 processor, and the computations for \eqref{eq:numerical-bound2}--\eqref{eq:numerical-bound4} were performed
were performed on a 3.4 GHz Intel Core i7 processor.}
\label{table:benchmarking}
\end{center}
\end{table}

Additional details on \texttt{SofaBounds} can be found in the documentation included with the package.

\end{document}